\documentclass[12pt]{amsart}
\font\smallit=cmti10
\font\smalltt=cmtt10
\font\smallrm=cmr9

\usepackage{amssymb}
\theoremstyle{plain}
\newtheorem{theorem}{Theorem}
\newtheorem{corollary}{Corollary}

\newtheorem{proposition}{Proposition}
\theoremstyle{example}

\theoremstyle{definition}

\theoremstyle{remark}

\numberwithin{equation}{section}

\setbox0=\hbox{$+$}
\newdimen\plusheight
\plusheight=\ht0
\def\+{\;\lower\plusheight\hbox{$+$}\;}

\setbox0=\hbox{$-$}
\newdimen\minusheight
\minusheight=\ht0
\def\-{\;\lower\minusheight\hbox{$-$}\;}

\setbox0=\hbox{$\cdots$}
\newdimen\cdotsheight
\cdotsheight=\plusheight
\def\cds{\lower\cdotsheight\hbox{$\cdots$}}
\begin{document}
\begin{center}
{\bf    COMBINATORIAL IDENTITIES DERIVING FROM THE $n$-th POWER OF A
$2\times 2$ MATRIX }
\vskip 20pt
{\bf J. Mc Laughlin \footnote{ http://www.trincoll.edu/~jmclaugh/}}\\
{\smallit   Mathematics Department
 Trinity College
300 Summit Street, Hartford, CT 06106-3100}\\
{\tt   james.mclaughlin@trincoll.edu}\\
\vskip 10pt
\end{center}
\vskip 30pt
\centerline{\smallit Received: , Accepted: , Published: }
\vskip 30pt
\centerline{\bf Abstract}
\noindent
In this paper we give a new  formula for the $n$-th power of a
$2\times2$ matrix.

More precisely, we prove the following: Let $A= \left (
\begin{matrix}
a & b \\
c & d
\end{matrix}
\right )$ be an arbitrary $2\times2$ matrix, $T=a+d$  its trace,
 $D= ad-bc$ its determinant and define
\[
y_{n} :\,= \sum_{i=0}^{\lfloor n/2 \rfloor}\binom{n-i}{i}T^{n-2
i}(-D)^{i}.
\]
Then, for $n \geq 1$,
\begin{equation*}
A^{n}=\left (
\begin{matrix}
y_{n}-d \,y_{n-1} & b \,y_{n-1} \\
c\, y_{n-1}& y_{n}-a\, y_{n-1}
\end{matrix}
\right ).
\end{equation*}

We use this formula together with an existing formula for the
$n$-th power of a matrix,  various matrix identities,  formulae
for the $n$-th power of particular matrices, etc, to derive
various combinatorial identities.

\pagestyle{myheadings}
\markright{\smalltt INTEGERS:
 \smallrm ELECTRONIC JOURNAL OF COMBINATORIAL NUMBER THEORY
\smalltt x (200x), \#A11\hfill}
\thispagestyle{empty}
\baselineskip=15pt
\vskip 30pt






\section{Introduction } \label{S:intro}
Throughout the paper, let $I$ denote the $2\times2$-identity
matrix and $n$ an arbitrary positive integer.
 In \cite{W92},
Williams gave the following formula for the $n$-th power of
$2\times2$ matrix $A$ with eigenvalues $\alpha$ and $\beta$:
\begin{equation}\label{eq1}
A^{n}=
\begin{cases}
\alpha^{n} \left ( \frac{A-\beta I}{\alpha-\beta} \right) +
\beta^{n} \left ( \frac{A-\alpha I}{\beta-\alpha} \right),
\,\,\,\, \text{ if } \alpha \not = \beta,\\
\phantom{as} \\ \alpha^{n-1}(nA-(n-1)\alpha I),
\,\,\,\,\,\,\,\,\,\,\text{ if } \alpha = \beta.
\end{cases}
\end{equation}
Blatz had given a similar, if slightly more complicated
expression, in \cite{B68}.

For our present purposes (producing combinatorial identities), it
is preferable to express $A^{n}$ directly in terms of the entries
of $A$, rather than the eigenvalues of $A$. If we let $T$ denote
the trace of $A$ and $D$ its determinant, then, without loss of
generality, $\alpha = (T+\sqrt{T^{2}-4\,D})/2$ and $\beta =
(T+\sqrt{T^{2}-4\,D})/2$. A little elementary arithmetic gives
that (in the case $\alpha \not = \beta$), if
\begin{equation}\label{eq0}
z_{n}:=\frac{\alpha^{n}-\beta^{n}}{\alpha-\beta}=
\frac{\sum_{m=0}^{\lfloor (n-1)/2
\rfloor}\binom{n}{2m+1}T^{n-2m-1}(T^{2}-4D)^{m}}{2^{n-1}},
\end{equation}
then
\begin{equation}\label{eq1alt}
A^{n}= z_{n}A-z_{n-1}\,D\,I.
\end{equation}
Equation \ref{eq1alt} also holds in the case of equal eigenvalues
(when $T^{2}-4\,D=0$), if $z_{n}$ is assumed to have the value on
the right of \eqref{eq0}. The key point here is that if another
closed-form expression exists for
 $A^{n}$, then equating the expressions for the entries of $A^{n}$
 from this closed form with the expressions derived from
 Equations \ref{eq0} and  \ref{eq1alt} will produce various identities.

 As an illustration, we consider the following  example.
 Let $\{F_{n}\}_{n=1}^{\infty}$ denote the Fibonacci sequence,
defined by $F_{0}=0$, $F_{1}=1$ and $F_{i+1}=F_{i}+F_{i-1}$, for
$i\geq 1$.  The identity
\begin{equation}\label{eq2a}
\left (
\begin{matrix}
1 & 1 \\
1 & 0
\end{matrix}
\right )^{n} = \left (
\begin{matrix}
F_{n+1} & F_{n} \\
F_{n} & F_{n-1}
\end{matrix}
\right )
\end{equation}
leads directly to several of the known formulae involving the
Fibonacci numbers. For example, direct substitution of $T=-D = 1$
in Equation \ref{eq0} gives Formula 91 from Vajda's list
\cite{V89}:
\begin{equation}\label{eqvaj}
F_{n} =\frac{\sum_{m=0}^{\lfloor (n-1)/2
\rfloor}\binom{n}{2m+1}5^{m}}{2^{n-1}}.
\end{equation}
Many other formulae can be similarly derived. This method of
deriving combinatorial identities of course needs \emph{two}
expressions for $A^{n}$ and for a general $2 \times 2$ matrix,
these have not been available.

In this present paper we remedy this situation by presenting a new
formula for $A^{n}$, where the entries are also expressed in terms
of the entries in $A$.

\vskip 30pt

\section{Main Theorem}
We prove the following theorem.
\begin{theorem}\label{t1}
Let
\begin{equation}\label{eqt1a}
A= \left (
\begin{matrix}
a & b \\
c & d
\end{matrix}
\right )
\end{equation}
be an arbitrary $2\times2$ matrix and let $T=a+d$ denote its trace
and $D= ad-bc$ its determinant. Let
\begin{equation}\label{eqt1b}
y_{n} = \sum_{i=0}^{\lfloor n/2 \rfloor}\binom{n-i}{i}T^{n-2
i}(-D)^{i}.
\end{equation}
Then, for $n \geq 1$,
\begin{equation}\label{eqt1c}
A^{n}=\left (
\begin{matrix}
y_{n}-d \,y_{n-1} & b \,y_{n-1} \\
c\, y_{n-1}& y_{n}-a\, y_{n-1}
\end{matrix}
\right ).
\end{equation}
\end{theorem}

\begin{proof}
The proof is by induction. Equation \ref{eqt1c} is easily seen to
be true for $n=1,\,2$, so suppose it is true for $n=1,\dots , k$.
This implies  that
{\allowdisplaybreaks
\begin{equation*}
A^{k+1}=\left (
\begin{matrix}
a\,y_{k}+(b\,c-a\,d)y_{k-1} & b \,y_{k} \\
c\, y_{k}& d\,y_{k}+(b\,c-a\,d)y_{k-1}
\end{matrix}
\right ).
\end{equation*}
}
Thus the result will follow if it can be shown that
\begin{equation*}
y_{k+1}=(a+d)y_{k}+(b\,c-a\,d)y_{k-1}.
\end{equation*}
Upon substituting from Equation \ref{eqt1b}, we have that
{\allowdisplaybreaks
\begin{align*}
(a+d)y_{k}&+(b\,c-a\,d)y_{k-1}\\
&=
\sum_{i=0}^{\lfloor k/2
\rfloor}\binom{k-i}{i}T^{k+1-2 i}(-D)^{i}
+ \sum_{i=0}^{\lfloor (k-1)/2
\rfloor}\binom{k-1-i}{i}T^{k-1-2 i}(-D)^{i+1}\\
&= \sum_{i=0}^{\lfloor k/2
\rfloor}\binom{k-i}{i}T^{k+1-2 i}(-D)^{i}
+ \sum_{i=1}^{\lfloor (k+1)/2
\rfloor}\binom{k-i}{i-1}T^{k+1-2 i}(-D)^{i}.
\end{align*}
}
If $k$ is even, then $\lfloor k/2 \rfloor = \lfloor (k+1)/2
\rfloor$, and
{\allowdisplaybreaks
\begin{align*}
(a+d)y_{k}+(b\,c-a\,d)y_{k-1}
&=\binom{k}{0}T^{k+1}
 + \sum_{i=1}^{\lfloor (k+1)/2 \rfloor}\left (\binom{k-i}{i}+
\binom{k-i}{i-1} \right )T^{k+1-2 i}(-D)^{i}\\
&=\sum_{i=0}^{\lfloor (k+1)/2 \rfloor}\binom{k+1-i}{i}T^{k+1-2
i}(-D)^{i}=y_{k+1}.
\end{align*}
}
 If $k$ is odd, then $\lfloor k/2 \rfloor = \lfloor (k-1)/2
\rfloor$, and
{\allowdisplaybreaks
\begin{align*}
(a+d)y_{k}+(b\,c-a\,d)y_{k-1}&=\binom{k}{0}T^{k+1}
 + \sum_{i=1}^{\lfloor (k-1)/2 \rfloor}\left (\binom{k-i}{i}+
\binom{k-i}{i-1} \right )T^{k+1-2 i}(-D)^{i}\\
& \phantom{saaasddsad}+ \binom{k-\lfloor
(k+1)/2 \rfloor}{\lfloor (k+1)/2 \rfloor-1}T^{k+1-2 \lfloor
(k+1)/2 \rfloor}(-D)^{\lfloor (k+1)/2 \rfloor }\\
&=\sum_{i=0}^{\lfloor (k+1)/2 \rfloor}\binom{k+1-i}{i}T^{k+1-2
i}(-D)^{i}=y_{k+1}.
\end{align*}
}
This completes the proof.
\end{proof}
Remark:
 Schwerdtfeger (\cite{S62}, pages 104--105) outlines a method due to Jacobsthal \cite{J19}
for finding the $n$ power of an arbitrary $2\times2$ matrix $A$ which,
after a little manipulation, is essentially
equivalent to the method described in Theorem \ref{t1}. We were not aware of Jacobsthal's
result
before our own discovery.

The aim now becomes to use  combinations of the formulae at
\eqref{eq1alt} and \eqref{eqt1c}, together with various devices
such writing $A^{nk}=(A^{k})^{n}$, writing $A=B+C$ or $A=B\,C$
where $B$ and $C$ commute, etc, to produce combinatorial
identities. One can also consider particular matrices $A$ whose
$n$-th power has a simple form, and then use the formulae at
\eqref{eq1alt} and \eqref{eqt1c} to derive combinatorial
identities,

As an immediate consequence of \eqref{eq1alt} and \eqref{eqt1c},
we have the following corollary to Theorem \ref{t1}.
\begin{corollary}\label{cor1}
For $1 \leq j \leq \lfloor (n-1)/2 \rfloor$,
{\allowdisplaybreaks
\begin{align}\label{eqt1d}
\binom{n}{2j+1} &=\sum_{i=j}^{\lfloor (n-1)/2
\rfloor}(-1)^{i-j}2^{n-1-2i}\binom{i}{j}\binom{n-1-i}{i},\\
\binom{n-1-j}{j}&=2^{-n+1+2j}\sum_{i=j}^{\lfloor (n-1)/2
\rfloor}\binom{n}{2i+1} \binom{i}{j}. \notag
\end{align}
}
\end{corollary}
\begin{proof}
For the first identity, we equate the $(1,2)$ entries at \eqref{eq1alt}
and \eqref{eqt1c}
and then let $T=2x$ and $D=x^2-y$.  This gives that
{\allowdisplaybreaks
\begin{equation*}
\sum_{i=0}^{\lfloor (n-1)/2 \rfloor}\binom{n-1-i}{i}(2x)^{n-1-2
i}(y-x^{2})^{i} = \sum_{i=0}^{ \lfloor (n-1)/2 \rfloor}
\binom{n}{2\,i+1} x^{n-1-2i} y^{i}.
\end{equation*}
}
We now expand the left side and equate coefficients of equal
powers of $y/x^2$.

For the second identity, we similarly equate the $(1,2)$ entries at \eqref{eq1alt}
and \eqref{eqt1c}, expand $(T^{2}-4 D)^{m}$ by the binomial theorem and
compare coefficients of like powers of $(-D/T^{2})$.

\end{proof}
Remark: The literature on combinatorial identities is quite extensive and all
of the identities we describe in this paper may already exist
elsewhere. However, we believe that at least some of the methods
used to generate/prove them are new.


\section{Some Formulae for the Fibonacci Numbers}

 The identity (\cite{V61}, pp. 18-20)
{\allowdisplaybreaks
\begin{equation}\label{eq2b}
F_{n}=\sum_{i=0}^{\lfloor (n-1)/2 \rfloor}\binom{n-1-i}{i}
\end{equation}
}
follows from \eqref{eq2a}, \eqref{eqt1b} and \eqref{eqt1c}, upon
setting $T=-D=b=1$.

The Lucas sequence $\{L_{n}\}_{n=1}^{\infty}$ is defined by
$L_{1}=1$, $L_{2}=3$ and $L_{n+1}=L_{n}+L_{n-1}$, for $n \geq 2$.
The identity $A^{n\,k}=(A^{n})^{k}$ implies
{\allowdisplaybreaks
\begin{equation*}
\left (
\begin{matrix}
F_{n\,k+1} & F_{n\,k} \\
F_{n\,k} & F_{n\,k-1}
\end{matrix}
\right ) =\left (
\begin{matrix}
F_{n+1} & F_{n} \\
F_{n} & F_{n-1}
\end{matrix}
\right )^{k}.
\end{equation*}
}
The facts that $L_{n}=F_{n+1}+F_{n-1}$ and
$F_{n+1}F_{n-1}-F_{n}^{2}=(-1)^{n}$ together with Theorem \ref{t1}
now imply the following identity from \cite{J03} (page 5):
{\allowdisplaybreaks
\begin{equation}\label{eq2c}
F_{n\,k}=F_{n}\sum_{i=0}^{\lfloor (k-1)/2
\rfloor}\binom{k-1-i}{i}L_{n}^{k-1-2 i}(-1)^{i(n+1)}.
\end{equation}
}
Let $\imath:=\sqrt{-1}$ and  define
$B:=
\frac{1}{\sqrt{2\,i-1}}
\left (
\begin{matrix}
1+\imath & \imath \\
\imath & 1
\end{matrix}
\right )$.
Then $B^{2}=
\left (
\begin{matrix}
1 & 1 \\
1 & 0
\end{matrix}
\right )$,  leading to the identity:
{\allowdisplaybreaks
\begin{equation}\label{eq2d}
F_{k} =\frac{1}{\imath^{k-1}}
\sum_{m=0}^{k-1}\binom{2k-1-m}{m}(2+i)^{k-1- m}(-1)^{m},
\end{equation}
}
a variant of Ram's formula labelled FeiPi at \cite{R03}.

Many similar identities for the Fibonacci and Lucas sequences
can also be easily derived from Theorem \ref{t1}.

\vskip 30pt

\section{ A Binomial Expansion from Williams' Formula}

\begin{proposition}
Let $n \in \mathbb{N}$ and  $1 \leq t \leq n $, $t$ integral. Then
{\allowdisplaybreaks
\begin{multline}\label{eq4a}
\binom{n}{t}=
\sum_{m=0}^{\lfloor \frac{n-1}{2} \rfloor}
\sum_{j=0}^{n-1-2m}(-1)^{m}2^{n-1-2m-j}
\binom{n-1-m}{m}\binom{n-1-2m}{j}\binom{m}{t-j-1}.
\end{multline}
}
\end{proposition}
\begin{proof}
Let $A=\left (
\begin{matrix}
f+2 e & 1 \\
0 & f
\end{matrix}
\right )$. The eigenvalues of $A$ are clearly $f+2e$ and $f$. With
the notation of Theorem \ref{t1}, $T=2(e+f)$ and $D=f(f+2e)$.  If
we now equate the $(1,2)$ entry from the left side of Equation
\ref{eq1} with the $(1,2)$ entry from the left side of Equation
\ref{eqt1c}, we  get, after a little manipulation, that
{\allowdisplaybreaks
\begin{equation*}
(f+2e)^{n} = f^{n} + 2e
\sum_{m=0}^{\lfloor (n-1)/2
\rfloor}\binom{n-1-m}{m}(2(e+f))^{n-1-2 m}\left (- f(2e+f)\right )^{m}.
\end{equation*}
}
Next, replace $e$ by $e/2$ to get that
{\allowdisplaybreaks
\begin{equation}\label{eq4c}
(f+e)^{n} = f^{n} + e \sum_{m=0}^{\lfloor (n-1)/2
\rfloor}\binom{n-1-m}{m}(e+2f)^{n-1-2 m}\left (- f(e+f)\right
)^{m}.
\end{equation}
}
 Finally, set $f=1$, expand both sides of Equation \ref{eq4c} and
compare coefficients of like powers of $e$ to get the result.
\end{proof}

\vskip 30pt

\section{Commutating Matrices I}

As before, let
\[
A=\left (
\begin{matrix}
a & b \\
c & d
\end{matrix}
\right ).
\]
We write
{\allowdisplaybreaks
\begin{equation}\label{eqxa}
A=(mA+w\,I)+((1-m)A-w\,I),
\end{equation}
}
noting that the matrices on the right commute.
{\allowdisplaybreaks
\begin{equation}\label{eqx}
A^{n}
=
\sum_{j=0}^{n}
\binom{n}{j}(mA+w\,I) ^{n-j}((1-m)A-w\,I)^{j}
\end{equation}
}
and we might hope to use the formulae at \eqref{eq1} and
\eqref{eqt1a} to derive new  identities. We give two examples.

\begin{proposition}\label{prop2-}
Let $n \in \mathbb{N}$ and $k$, $r$ integers with $0 \leq k\leq r  \leq n$. Then
 {\allowdisplaybreaks
\begin{equation}\label{eqxaa}
(-1)^{n-k}\sum_{j=0}^{n-k}
 \binom{n}{j}
 \binom{n-j}{k}
\binom{j}{r-k}
(-1)^{j}= \begin{cases}\displaystyle{\binom{n}{k}}, & r=n,\\
\phantom{as} &\phantom{as}\\
0, & r \not = n.
\end{cases}
\end{equation}
}
\end{proposition}
\begin{proof}
Expand the right side of Equation \ref{eqx}  and re-index to get that
{\allowdisplaybreaks
\begin{multline*}
A^{n} = \sum_{r=0}^{n}w^{n-r}A^{r}(1-m)^{r}(-1)^{r}
\sum_{k=0}^{n}\left (
\frac{m}{1-m} \right )^{k}(-1)^{k} \\
\times \sum_{j=0}^{n-k} \binom{n}{j}\binom{n-j}{k}\binom{j}{r-k}(-1)^{j}.
\end{multline*}
}
If we compare coefficients of $w$ on both sides, it is clear that
{\allowdisplaybreaks
\begin{equation*}
\sum_{k=0}^{n}\left (
\frac{m}{1-m} \right )^{k}(-1)^{k}
 \sum_{j=0}^{n-k} \binom{n}{j}\binom{n-j}{k}\binom{j}{r-k}(-1)^{j}
=\begin{cases}(m-1)^{-n}, & r=n,\\
\phantom{as} &\phantom{as}\\
0, & r \not = n.
\end{cases}
\end{equation*}
}
Equivalently,
{\allowdisplaybreaks
\begin{equation*}
\sum_{k=0}^{n}
m^{k}(1-m)^{n-k}(-1)^{n-k}
 \sum_{j=0}^{n-k} \binom{n}{j}\binom{n-j}{k}\binom{j}{r-k}(-1)^{j}
=\begin{cases}1, & r=n,\\
\phantom{as} &\phantom{as}\\
0, & r \not = n.
\end{cases}
\end{equation*}
}
Finally, equating coefficients of powers of $m$ on both sides of the last
equality give the result. We note that the case for $r=n$ follows since the left side then
must equal
$(m+(1-m))^n$.
\end{proof}

\begin{proposition}\label{prop2}
Let $n \in \mathbb{N}$ and $t$ an integer such that $0 \leq t \leq n$. For each integer $
i \in [0,\,\,
 \lfloor (n-1)/2\rfloor ]$,
 {\allowdisplaybreaks
\begin{align}\label{eqxab}
\sum_{j=1}^{n} &\sum_{r=i}^{\lfloor (j-1)/2\rfloor}
\sum_{k=0}^{j-1-2r}
 \binom{n}{j}
 \binom{j-1-r}{r}
\binom{j-1-2r}{k}
\binom{r}{i}\\
&\times \binom{r-i}{j-k-r+i-n+t} (-1)^{r+t+j+n+i} 2^{k}
= \begin{cases}\binom{n-1-i}{i}, & t=0,\\
\phantom{as} &\phantom{as}\\
0, & t \not = 0.
\end{cases}\notag
\end{align}
}
\end{proposition}
\begin{proof}
We set $m=0$ in Equations \ref{eqxa}and \ref{eqx}, so that
{\allowdisplaybreaks
\begin{equation*}
\left (
\begin{matrix}
a & b \\
c & d
\end{matrix}
\right )^{n}
=
\sum_{j=0}^{n}
\binom{n}{j} w^{n-j}
\left (
\begin{matrix}
a-w & b  \\
c   & d-w
\end{matrix}
\right )^{j}.
\end{equation*}
}
If we use Theorem \ref{t1} and compare values in the $(1,2)$ position, we
deduce that
{\allowdisplaybreaks
\begin{multline}\label{eq5b}
 \sum_{i=0}^{\lfloor (n-1)/2 \rfloor}\binom{n-1-i}{i}(a+d)^{n-1-2
i}(bc-ad)^{i}=\\
\sum_{j=1}^{n}
\binom{n}{j} w^{n-j}\sum_{r=0}^{\lfloor (j-1)/2 \rfloor}\binom{j-1-r}{r}(a+d-2w)^{j-1-2
r}(bc-(a-w)(d-w))^{r}.
\end{multline}
}
For ease of notation we replace $a+d$ by $T$ and $ad-bc$ by $D$.
If we then expand the right side and collect powers of $w$, we get
that
{\allowdisplaybreaks
\begin{multline*}
 \sum_{i=0}^{\lfloor (n-1)/2 \rfloor}\binom{n-1-i}{i}T^{n-1-2
i}(-D)^{i}=\\
\sum_{j=1}^{n}
\sum_{r=0}^{\lfloor (j-1)/2\rfloor}
\sum_{k=0}^{j-1-2r}
\sum_{m=0}^{r}
\sum_{p=0}^{m}
\binom{n}{j}
 \binom{j-1-r}{r}
\binom{j-1-2r}{k}
\binom{r}{m}
\binom{m}{p} \\
\times(-1)^{k+r-m+p} 2^{k} T^{j-1-2r-k+m-p} D^{r-m} w^{n-j+k+m+p}.
\end{multline*}
}
If we solve $t=n-j+k+m+p$ for $p$, we have that
 {\allowdisplaybreaks
\begin{align*}
\sum_{i=0}^{\lfloor (n-1)/2 \rfloor}&\binom{n-1-i}{i}T^{n-1-2
i}(-D)^{i} \\
&=\sum_{t=0}^{n-1}w^{t} \sum_{j=1}^{n} \sum_{r=0}^{\lfloor (j-1)/2\rfloor}
\sum_{k=0}^{j-1-2r} \sum_{m=0}^{r}
 \binom{n}{j}
 \binom{j-1-r}{r}
\binom{j-1-2r}{k}
\binom{r}{m}\notag \\
&\times \binom{m}{j-k-m-n+t} (-1)^{r+t+j+n} 2^{k} T^{2m+n-t-2r-1}
D^{r-m}.\notag
\end{align*}
} Next, cancel a factor of $T^{n-1}$ from both sides, replace $T$
by $1/y$, $D$ by $-x/y^{2}$ and $m$ by $r-i$ to get that
{\allowdisplaybreaks
\begin{align*}
\sum_{i=0}^{\lfloor (n-1)/2 \rfloor}&\binom{n-1-i}{i}x^{i} \\
&=\sum_{t=0}^{n-1}(wy)^{t}\sum_{j=1}^{n} \sum_{r=0}^{\lfloor (j-1)/2\rfloor}
\sum_{k=0}^{j-1-2r} \sum_{i=0}^{r}
 \binom{n}{j}
 \binom{j-1-r}{r}
\binom{j-1-2r}{k}
\binom{r}{r-i}\notag \\
&\phantom{aasdasdasdasdsasd}\times \binom{r-i}{j-k-r+i-n+t} (-1)^{r+t+j+n+i} 2^{k}
x^{i}\notag \\
&\phantom{as}\notag\\
&=\sum_{t=0}^{n-1}(wy)^{t}\sum_{i=0}^{\lfloor (n-1)/2\rfloor}\sum_{j=1}^{n}
\sum_{r=i}^{\lfloor (j-1)/2\rfloor} \sum_{k=0}^{j-1-2r}
 \binom{n}{j}
 \binom{j-1-r}{r}
\binom{j-1-2r}{k}
\binom{r}{i}\notag \\
&\phantom{aasdsadaadsasddasd}\times \binom{r-i}{j-k-r+i-n+t} (-1)^{r+t+j+n+i} 2^{k}
x^{i}\notag \\
&\phantom{as}\notag
\end{align*}
}
Finally, compare coefficients of $x^{i}(wy)^{t}$ to get the result.
\end{proof}

\begin{corollary}
For every complex number $w$ different from 0, 1/2, $\phi$ (the golden ratio)
 or $-1/\phi$,
\begin{equation}\label{eq5g}
F_{n}= \sum_{j=1}^{n}\sum_{r=0}^{\lfloor (j-1)/2 \rfloor}
\binom{n}{j} \binom{j-1-r}{r}w^{n-j}(1-2w)^{j-1-2
r}(1+w-w^{2})^{r}.
\end{equation}
\end{corollary}
\begin{proof}
Let $a=b=c=1$, $d=0$ in Equation \ref{eq5b} and once again use the
identity at \eqref{eq2b}.
\end{proof}
This identity contains several of the known identities for the Fibonacci sequence
as special cases. For example, letting  $w \to \phi$  gives Binet's Formula
 (after summing using the Binomial Theorem);
$w \to 0$ gives the formula at \eqref{eq2b}; $w \to 1/2$ gives \eqref{eqvaj}; $w=1$
gives the identity
\begin{equation*}
F_{n}= \sum_{j=1}^{n}(-1)^{j-1}\binom{n}{j} F_{j},
\end{equation*}
which is a special case of Formula 46 from Vajda's list
(\cite{V89}, page 179).

\vskip 30pt

\section{Commutating Matrices II}

Let
$A=\left (
\begin{matrix}
a & b  \\
c   & d
\end{matrix}
\right )$ with $T=a+d$ and $D=ad-bc$  and $I$ the $2\times2$ identity matrix as before.
 Let $g$ be an arbitrary real or complex number such that $g^{2}+Tg+D \not = 0$.
It is easy to check that
{\allowdisplaybreaks
\begin{equation}\label{eq61}
A=\frac{1}{g^{2}+Tg+D}(A+g\,I)(g\,A+D\,I),
\end{equation}
}
and that $A+g\,I$ and $g\,A+D\,I$ commute.
If both sides of Equation \ref{eq61} are now raised to the $n$-th power, the right side
expanded via the Binomial Theorem and powers of $A$ collected, we get the following proposition:

\begin{proposition}\label{prop3}
Let $A$ be an arbitrary $2\times2$ matrix with trace $T$ and
determinant $D\not = 0$. Let $g$ be a complex number such that
$g^{2}+Tg+D \not = 0$, $g \not =0$  and let $n$ be a positive
integer. Then
{\allowdisplaybreaks
\begin{equation}\label{eq62}
A^{n}=\left (\frac{g\,D}{g^{2}+Tg+D}\right )^{n}
\sum_{r=0}^{2n}
\sum_{i=0}^{r}
\binom{n}{i}
\binom{n}{r-i}
\left(\frac{D}{g^{2}}\right)^{i}
\left (\frac{g}{D}\right)^{r}
 A^{r}.
\end{equation}
}
\end{proposition}

\begin{corollary}
Let $n$ be a positive integer and let $m$ be an integer with $0
\leq m \leq 2n$. Then for $ -n \leq w \leq n $,
\begin{multline}\label{eq63}
\sum_{k=0}^{n-1} \binom{n-1-k}{k} \binom{n}{w+k}
\binom{k+w}{m-k-w}(-1)^{k}=\\
 \sum_{k=-2w - n + m + 1}^{m - w} \binom{n}{k+w}
\binom{n}{n+k+w-m} \binom{k+n+2w-m-1}{k}(-1)^{k} .
\end{multline}
\end{corollary}
Remark: Consideration of the binomial coefficients on either side
of \eqref{eq63} shows that this identity is non-trivial only for
$m/2-\lfloor (n-1)/2 \rfloor \leq w \leq m$, but we prefer to
state the limits on $w$ as we have done for neatness reasons.
\begin{proof}
If  Equation \ref{eq62} is multiplied by $(g^{2}+Tg+D)^{n}$, both
sides  expanded  and coefficients of like powers of $g$ compared,
we get that for each integer $m \in [0,\,\,2 n]$,
{\allowdisplaybreaks
\begin{multline}\label{eq6a}
A^{n} \sum_{j=0}^{n}\binom{n}{j}\binom{j}{m-j}D^{n-j}T^{2j-m}=\\
\sum_{r=0}^{2n}\frac{1+(-1)^{r+m+n}}{2}
\binom{n}{\frac{r-m+n}{2}}\binom{n}{\frac{r+m-n}{2}}
D^{(3n-r-m)/2}A^{r}.
\end{multline}
}
If we next assume that $(1,2)$ entry of $A$ is non-zero, apply
Theorem \ref{t1} to the various powers of $A$ on each side of
Equation \ref{eq6a} and compare the $(1,2)$ entries on each side
of \eqref{eq6a}, we get that
{\allowdisplaybreaks
\begin{multline}\label{eq6b}
\sum_{k=0}^{\lfloor (n-1)/2
\rfloor}\binom{n-1-k}{k}T^{n-1-2k}(-D)^{k}
 \sum_{j=0}^{n}\binom{n}{j}\binom{j}{m-j}D^{n-j}T^{2j-m}=\\
\sum_{r=0}^{2n}\frac{1+(-1)^{r+m+n}}{2}
\binom{n}{\frac{r-m+n}{2}}\binom{n}{\frac{r+m-n}{2}}
D^{(3n-r-m)/2}
 \sum_{k=0}^{\lfloor \frac{r-1}{2}
\rfloor}\binom{r-1-k}{k}T^{r-1-2k}(-D)^{k}.
\end{multline}
}
Equivalently, after re-indexing and changing the order of
summation,
 {\allowdisplaybreaks
\begin{multline}\label{eq6c}
\sum_{v=0}^{2n}\sum_{k=0}^{n}
 \binom{n-1-k}{k}
 \binom{n}{k+n-v}
 \binom{k+n-v}{m-k-n+v}
 (-1)^{k}
 \left (
 \frac{D}{T^{2}}
 \right)^{v}\\
=\sum_{v=0}^{2n} \sum_{k= 2v-3n +m+1 }^{ v+m-n } \binom{n}{k-v+n}
\binom{n}{k-v+2n-m} \\
\times \binom{k-2v+3n-m-1}{k}(-1)^{k} \left
( \frac{D}{T^{2}} \right)^{v}.
\end{multline}
} The result follows from comparing coefficients of equal powers
of $D/T^{2}$ and replacing $v$ by $n-w$.
\end{proof}

\begin{corollary}\label{corspec}
Let $n$ be a positive integer and let $m \in \{0,\,1\,\dots ,
2n\}$. Then
{\allowdisplaybreaks
\begin{equation}\label{eq6d1}
\sum_{j=0}^{n}\binom{n}{j}\binom{j}{m-j}2^{2j-m}=
\sum_{\substack{r=0,\\
r+m+n\,\, even}}^{2n}
\binom{n}{\frac{r-m+n}{2}}\binom{n}{\frac{r+m-n}{2}},
\end{equation}
}
{\allowdisplaybreaks
\begin{equation}\label{eq6d2}
(n+1)\sum_{j=0}^{n}\binom{n}{j}\binom{j}{m-j}2^{2j-m}=
\sum_{\substack{r=0,\\
r+m+n\,\, even}}^{2n}
\binom{n}{\frac{r-m+n}{2}}\binom{n}{\frac{r+m-n}{2}}(r+1),
\end{equation}
}
{\allowdisplaybreaks
\begin{multline}\label{eq6d3}
(2^{n+1}-1)\sum_{j=0}^{n}\binom{n}{j}\binom{j}{m-j}2^{n-j}3^{2j-m}=\\
\sum_{\substack{r=0,\\
r+m+n\,\, even}}^{2n}
\binom{n}{\frac{r-m+n}{2}}\binom{n}{\frac{r+m-n}{2}}2^{(3n-r-m)/2}(2^{r+1}-1),
\end{multline}
}
and
{\allowdisplaybreaks
\begin{equation}\label{eq6d4}
F_{n+2}\sum_{j=0}^{n}\binom{n}{j}\binom{j}{m-j}(-1)^{m+j}=
\sum_{\substack{r=0,\\
r+m+n\,\, even}}^{2n}
\binom{n}{\frac{r-m+n}{2}}\binom{n}{\frac{r+m-n}{2}}(-1)^{(-n+r-m)/2}F_{r+2}.
\end{equation}
}
\end{corollary}
\begin{proof}
Equations \eqref{eq6d1} to \eqref{eq6d4} follow from comparing
$(1,1)$ entries at \eqref{eq6a}, using, respectively, the
following identities:
 {\allowdisplaybreaks
\begin{align*}
I^{n}&=I,\\
&\phantom{as}\\
 \left (
\begin{matrix}
2 & 1  \\
-1   & 0
\end{matrix}
\right )^{n}&= \left (
\begin{matrix}
n+1 & n  \\
-n   & -n+1
\end{matrix}
\right ),\\
&\phantom{as}\\
\left (
\begin{matrix}
3 & 1  \\
-2   & 0
\end{matrix}
\right )^{n}&= \left (
\begin{matrix}
2^{n+1}-1 & 2^{n}-1  \\
-2^{n+1}+2   & -2^{n}+2
\end{matrix}
\right ),
\\&\phantom{as}\\
\left (
\begin{matrix}
-2 & -1  \\
1   & 1
\end{matrix}
\right )^{n}&= (-1)^{n}\left (
\begin{matrix}
F_{n+2} & F_{n}  \\
-F_{n}   & -F_{n-2}
\end{matrix}
\right ).
\end{align*}
}
\end{proof}
\begin{corollary}
If $g$ is a complex number  different from $0$, $-\phi$ or $1/\phi$, then
{\allowdisplaybreaks
\begin{equation}
F_{n} = \left (\frac{-g}{g^{2}+g-1}\right )^{n}
\sum_{r=0}^{2n}
\sum_{i=0}^{r}
\binom{n}{i}
\binom{n}{r-i}
\left ( -1\right)^{r+i}g^{r-2i}
 F_{r}.
\end{equation}
}
\end{corollary}
\begin{proof}
Let $A=\left (
\begin{matrix}
1 & 1  \\
1   & 0
\end{matrix}
\right )$ and compare the $(1,2)$ entries on each side of Equation
\ref{eq62}, using \eqref{eq2a}.
\end{proof}

\vskip 30pt

\section{Miscellaneous identities derived from particular matrices}
The identities in Corollary \ref{corspec} derived from the fact that the
$n$-th power of each of the various  matrices mentioned in the proof
 has a simple, elegant form.
It is of course easy to construct such  $2\times 2$  matrices
whose $n$-th power  has a  similar
"nice" form by constructing matrices $\left (
\begin{matrix}
a & b  \\
c   & d
\end{matrix}
\right )$ with a predetermined pair of eigenvalues $\alpha$ and $\beta$
and using the formula of Williams at \eqref{eq1}.  Theorem \ref{t1}
 can then be applied to any such matrix to produce   identities of the form
{\allowdisplaybreaks
\begin{equation*}
 \sum_{i=0}^{\lfloor (n-1)/2 \rfloor}\binom{n-1-i}{i}(a+d)^{n-1-2
i}(-(a d - b c))^{i}=
\begin{cases}
\frac{\alpha^{n}-\beta^{n}}{\alpha-\beta},
\,\,\,\, \,\,\,\,\,\text{ if } \alpha \not = \beta,\\
\phantom{as} \\
n \alpha^{n-1},
\,\,\,\,\,\,\,\,\,\,\text{ if } \alpha = \beta.
\end{cases}
\end{equation*}
}
As an example, the matrix
$\left (
\begin{matrix}
2 & 1  \\
-1   & 0
\end{matrix}
\right )$ has both eigenvalues equal to 1, leading to the following identity
for $n \geq 1$:
{\allowdisplaybreaks
\begin{equation}
n=\sum_{i=0}^{\lfloor (n-1)/2 \rfloor}\binom{n-1-i}{i}2^{n-1-2
i}(-1)^{i}.
\end{equation}
}

We give one  additional example.

\begin{proposition}
Let $n$ be a  positive integer, and $s$ an integer with $0 \leq s \leq n-1$. Then
{\allowdisplaybreaks
\begin{equation}
\sum_{k \geq 0}\binom{s}{k}\binom{n-k-1}{s-1}(-1)^{k} = 0.
\end{equation}
}
\end{proposition}
\begin{proof}
We consider
$\left (
\begin{matrix}
i & 0  \\
0   & j
\end{matrix}
\right )^{n}$ and use Theorem \ref{t1} to get that
{\allowdisplaybreaks
\begin{equation*}
i^{n}=\sum_{k=0}^{\lfloor n/2 \rfloor}\binom{n-k}{k}(i+j)^{n-2
k}(-i\,j)^{k}\\
-j\sum_{k=0}^{\lfloor (n-1)/2 \rfloor}\binom{n-1-k}{k}(i+j)^{n-1-2
k}(-i\,j)^{k}.
\end{equation*}
}
Upon expanding powers of $i+j$ and re-indexing, we get that
\begin{multline}\label{eqlas1}
i^{n}=i^{n} + \sum_{s=0}^{n-1}i^{s}j^{n-s}
\sum_{k}\left \{
\binom{n-k}{k}\binom{n-2k}{s-k} - \binom{n-1-k}{k}\binom{n-1-2k}{s-k}
\right \}(-1)^{k}.
\end{multline}
Finally, we note that
\[
\binom{n-k}{k}\binom{n-2k}{s-k} - \binom{n-1-k}{k}\binom{n-1-2k}{s-k}
=\binom{s}{k}\binom{n-k-1}{s-1}
\]
and compare coefficients of $i^{s}j^{n-s}$ on both sides of \eqref{eqlas1}
to get the result.
\end{proof}

\vskip 30pt

\section{concluding remarks}
Theorem \ref{t1} can be applied in a number of other ways. The
identity
\[
\left (
\begin{matrix}
\cos{n \theta} & \sin{n\theta}  \\
-\sin{n \theta}   & \cos{n \theta}
\end{matrix}
\right )=\left (
\begin{matrix}
\cos{ \theta} & \sin{\theta}  \\
-\sin{ \theta}   & \cos{ \theta}
\end{matrix}
\right )^{n}
\]
will give alternative (to those derived from De Moivre's Theorem) expressions
for $\cos{n \theta}$ and $\sin{n \theta}$.

If $(x_{n},\,y_{n})$ denotes the $n$-th largest pair of positive integers satisfying
the Pell equation $x^{2} - m y^{2}=1$, the identity
\[
\left (
\begin{matrix}
x_{n} &  m\,y_{n} \\
y_{n}   & x_{n}
\end{matrix}
\right )=\left (
\begin{matrix}
x_{1} &  m\,y_{1} \\
y_{1}   & x_{1}
\end{matrix}
\right )^{n}
\]
will give alternative expressions for $x_{n}$ and $y_{n}$ to those derived from
the formula $x_{n}+\sqrt{m}y_{n}=(x_{1}+\sqrt{m}y_{1})^{n}$.

Similarly, one can use Theorem \ref{t1} to
find various identities for the Brahmagupta polynomials \cite{S96}
$x_{n}$
and $y_{n}$  defined by
\[
\left (
\begin{matrix}
x_{n} &  \,y_{n} \\
t\, y_{n}   & x_{n}
\end{matrix}
\right )=\left (
\begin{matrix}
x_{1} &  y_{1} \\
t\,y_{1}   & x_{1}
\end{matrix}
\right )^{n}
\]
and the Morgan-Voyce polynomials  \cite{MV59} $B_{n}$ and $b_{n}$
defined by
\begin{align*}
\left (
\begin{matrix}
B_{n} &  -B_{n-1} \\
B_{n-1}   &- B_{n-2}
\end{matrix}
\right )=
\left (
\begin{matrix}
x+2 &  -1 \\
1   & 0
\end{matrix}
\right )^{n}
\end{align*}
and $b_{n}=B_{n}-B_{n-1}$.

It may be possible to extend some of the methods used in this paper
to develop new combinatorial identities. For example, it may be possible
to exploit the equation
\[
A=(m_{1}A+w_{1}I)+(m_{2}A+w_{2}I)+((1-m_{1}-m_{2})A-(w_{1}+w_{2})I)
\]
or, more generally, the equation
\[
A= \sum_{i=1}^{p}(m_{i}A+w_{i}I) + \left(\left (1- \sum_{i=1}^{p}m_{i} \right)A-
\sum_{i=1}^{p}w_{i}I\right )
\]
to find new combinatorial identities in a way similar to the use
 of Equation \ref{eqxa} in the section \emph{Commuting Matrices I}.
We hope to explore some of these ideas in a later paper.

 \allowdisplaybreaks{

}

\begin{thebibliography}{99}

 \bibitem{B68} P. J. Blatz,
\emph{On the Arbitrary Power of an Arbitrary ($2 \times 2$)-Matrix
(in Brief Versions).}
  The American Mathematical Monthly, Vol. \textbf{75}, No. 1.
(Jan., 1968), pp. 57-58.

\bibitem{G81} H. W. Gould,
\emph{A history of the Fibonacci $Q$-matrix and a higher-dimensional problem.}
 Fibonacci Quart. \textbf{19}  (1981), no. 3, 250--257.

\bibitem{J19} E. Jacobsthal,
\emph{Fibonaccische Polynome und Kreistheilungsgleichungen.}\\
Sitzungsberischte der berliner Math. Gesellschaft, \textbf{17}, (1919--20), 43--47.

\bibitem{J03} R. C. Johnson,
\emph{Matrix methods for Fibonacci and related sequences.} \\
http:// maths.dur.ac.uk/~dma0rcj/PED/fib.pdf, (August 2003).

\bibitem{MV59}
A. M. Morgan-Voyce,
 \emph{Ladder Network Analysis Using Fibonacci Numbers.}
 IRE Trans. Circuit Th. \textbf{CT-6}, 321-322, Sep. 1959.


\bibitem{R03} Rajesh Ram,
\emph{ Fibonacci number formulae,}
\\
http://users.tellurian.net/hsejar/maths/fibonacci/ (September 2003).

\bibitem{S62} Hans Schwerdtfeger,
\emph{Geometry of complex numbers.}
 Mathematical Expositions, No. 13 University of Toronto Press, Toronto 1962 xi+186 pp.

\bibitem{S96}
E. R. Suryanarayan,
 \emph{The Brahmagupta Polynomials.} Fib. Quart. \textbf{34}, 30-39, 1996.



\bibitem{V89} S. Vajda,
\emph{Fibonacci \& Lucas numbers, and the golden section.} Theory
and applications. Ellis Horwood Series: Mathematics and its
Applications. Ellis Horwood Ltd., Chichester; Halsted Press [John
Wiley \& Sons, Inc.], New York, 1989. 190 pp.


\bibitem{V61} N. N. Vorob'ev,
\emph{Fibonacci numbers.} Translated from the Russian by Halina
Moss; translation editor Ian N. Sneddon Blaisdell Publishing Co.
(a division of Random House), New York-London 1961 viii+66 pp.


\bibitem{W92} Kenneth S. Williams,
\emph{The nth Power of a $2 \times 2$ Matrix (in Notes).}
 Mathematics Magazine, Vol. \textbf{65}, No. 5. (Dec., 1992), p. 336.




\end{thebibliography}
\end{document}